\documentclass[11pt, reqno]{amsart}
\usepackage{amsmath}
\usepackage{amssymb}
\usepackage{esint}
\usepackage[hidelinks]{hyperref}
\newtheorem{theorem}{Theorem}[section]
\newtheorem{lemma}[theorem]{Lemma}
\newtheorem{proposition}[theorem]{Proposition}
\newtheorem{corollary}[theorem]{Corollary}

\theoremstyle{definition}
\newtheorem{definition}[theorem]{Definition}

\theoremstyle{remark}
\newtheorem{remark}[theorem]{Remark}
\newtheorem*{claim}{Claim}

\numberwithin{equation}{section}

\newcommand{\bb}[1]{\mathbb{#1}}
\newcommand{\symmtrix}{\mathcal{S}^{n\times n}}
\allowdisplaybreaks[4]

\title[A Generalization of Savin's small perturbation theorem]{ A Generalization of Savin's small perturbation theorem for fully nonlinear elliptic equations and applications}

\author{Zhenyu Fan}
\address{School of Mathematical Sciences, Peking University, Beijing, 100871, P. R. China}
\curraddr{}
\email{fanzhenyu@stu.pku.edu.cn}
\date{\today}

\begin{document}


\begin{abstract}
     In this note, we generalize Savin's small perturbation theorem \cite{Savin-07} to nonhomogeneous fully nonlinear equations $F(D^2u, Du, u,x)=f$ provided the coefficients and the right-hand side terms are H\"older small perturbations. As an application, we establish a partial regularity result for the sigma-$k$ Hessian equation $\sigma_{k}(D^2u)=f$.
\end{abstract}
\maketitle

\section{Introduction}
In this note, we study the interior $C^{2,\alpha}$ regularity of small perturbation  solutions to fully nonlinear equation 
 \begin{align}\label{main eq}
     F(D^2u, Du, u, x)=f(x) \qquad \mathrm{on}\ B_{1},
 \end{align}
 where $F:\mathcal{S}^{n\times n} \times \bb{R}^n \times \bb{R} \times B_{1}  \to \bb{R}$ is a smooth and locally uniformly elliptic operator and $f$ is a H\"older small perturbation. The locally uniform ellipticity is defined as following.

 \begin{definition}
     (i) $F$ is said to be \emph{elliptic}, if   
        \[F(M, p,z,x)\leq F(M+N, p,z,x), \quad \text{for}\ M,N\in\symmtrix\ \text{and}\ N\ge 0. \]
     
     (ii) $F$ is said to be \emph{uniformly elliptic}, if there exist constants $0<\lambda\leq \Lambda$, such that for any $M,N\in\symmtrix$ with $N\ge 0$.
     \[\lambda\|N\|\leq F(M+N,p,z,x)-F(M,p,z,x)\leq \Lambda\|N\|.\]

     (iii) $F$ is said to be \emph{locally uniformly elliptic} or \emph{$\rho\,$-uniformly elliptic}, if there exist constants $0<\lambda\leq \Lambda$, such that for any $ \|M\|,\|N\|, |p|, |z|\leq \rho$ with $N\ge 0$,
     \[\lambda\|N\|\leq F(M+N,p,z,x)-F(M,p,z,x)\leq \Lambda\|N\|. \]
 \end{definition}

\begin{remark}
    Locally uniform ellipticity or $\rho\,$-uniform ellipticity means that  $F$ is only uniformly elliptic a neighborhood of $\{ (0,0,0,x): x\in B_{1} \}$. If $\rho=\infty$, $\rho\,$-uniform ellipticity becomes uniform ellipticity.
\end{remark}

\begin{remark}
      Many important fully nonlinear PDEs, such as the Monge-Amp\`ere equation $\det D^2u=f$, the $\sigma_k\,$-Hessian equation $\sigma_{k}(D^2u)=f$, the $\sigma_{k}\,$-curvature equation $\sigma_{k}(\kappa(u))=f$, and the special Lagrangian equation $\mathrm{tr} \arctan D^2u =\Theta$,  are locally uniformly elliptic. This makes the study of locally uniformly elliptic equations particularly significant. 
\end{remark}

\begin{remark}\label{RMK: scaling cause problem}
    Locally uniformly elliptic equations have an important feature: scaling will change the uniform ellipticity. Take the simplest example, let $F=F(M)$ be a $\rho\,$-uniformly elliptic operator with ellipricity constant $0<\lambda\leq \Lambda$. and let $u$ be a solution to $F(D^2u)=0$. Consider the rescaled function $v=\delta u$, then $v$ solves
     \[ \widetilde{F}(D^2v):=\delta F\left(\dfrac{1}{\delta}D^2v\right)=0. \]
     Now the operator $\widetilde{F}(\cdot)=\delta F(\delta^{-1}\cdot)$ is $\delta\rho$-uniformly elliptic with same ellipticity constants. This obstacle is the major difficulty for developing the regularity theory. 
\end{remark}

    We aim to study the regularity of viscosity solutions to \eqref{main eq}, so we first recall the the definition of viscosity solutions for second order elliptic equations.
    \begin{definition}
        A continuous function $u$ is said to be a \emph{viscosity subsolution (resp. supersolution)} to \eqref{main eq} on $B_{1}$, if for any $C^2$ function $\varphi$ which touches $u$ from above (resp. below) at $x_{0}\in B_{1}$, we have
        \[  F(D^2\varphi(x_{0}), D\varphi(x_{0}), \varphi(x_{0}), x_{0} ) \ge\ \text{(resp. $\leq$)}\ f(x_{0}).  \]

        We say that $u$ is a \emph{viscosity solution} to \eqref{main eq} if it is both a viscosity subsolution and supersolution.
    \end{definition}
    
    We first briefly review the regularity theory of uniformly elliptic equations. Let $u$ be a viscosity solution to $F(D^2u)=0$ for uniformly elliptic and smooth $F$. In two dimensional case, Nirenberg \cite{Nirenberg-53} proved that $u$ is smooth. In higher dimensions, due to Krylov-Safonov's Harnack inequality \cite{Krylov-Safonov-1,Krylov-Safonov-2}, one can show that $u\in C^{1,\alpha}$ for some universal $\alpha$, see \cite[Chapter 4 and 5]{CCbook}. If we require $F$ is convex or concave in additional, Evans \cite{Evans} and Krylov \cite{Krylov} proved that $u\in C^{2,\alpha}$, hence smooth by Schauder theory and bootstrap argument, see also \cite[Chapter 6]{CCbook}. Later, many authors weakened the convexity/concavity condition in the Evans-Krylov theory. We refer to Caffarelli-Yuan \cite{Caffarelli-Yuan}, Caffarelli-Cabr\'e \cite{CC2003}, Collins \cite{Collins} and Goffi \cite{Goffi}. Moreover, If $F$ has the Liouville property, that is, every $C^{1,1}$ viscosity entire solution to $F(D^2u)=0$ with bounded Hessian must be a quadratic polynomial. Then the Evans-Krylov estimate still holds, that is, $u\in C^{2,\alpha}$ provided $u\in C^{1,1}$. This was proved by Huang \cite{Huang-2002} and Yuan \cite{Yuan-2001} via VMO type estimates. For general $F$, one even cannot expect the $C^{2}$ regularity, due to the counterexample of Nadirashivili-Vl\u adu\c t \cite{Nadirashivili}.

    Without the convexity or concavity condition, Savin \cite{Savin-07} first proved the following small perturbation theorem. 

    \begin{theorem}[Savin]\label{THM: Savin thm}
        For $0<\alpha<1$. Suppose $F: \symmtrix \times \bb{R}^n\times \bb{R}\times B_1\to \bb{R}$ satisfies 
        
    $\mathrm{H}1)$ $F$ is elliptic and $\rho\,$-uniformly elliptic with ellipticity constants $0<\lambda\leq \Lambda$;
    
    $\mathrm{H}2)$ $F(0,0,0,x)\equiv 0$.
    
    $\mathrm{H}3)$  $F\in C^2$ and $\|D^2F\|\leq K$ in a $\rho\,$-neighborhood of the set $\{(0,0,0,x): x\in B_1\}$.

    Let $u\in C(B_{1})$ be a viscosity solution to 
        \[F(D^2u, Du, u, x)=0\qquad \mathrm{on}\ B_{1}.\]
        There exist constants $\delta,C$, depending only on $n,\alpha, \rho,\lambda,\Lambda$ and $K$, such that if 
        \[ \|u\|_{L^{\infty}(B_{1})}\leq \delta, \]
         then $u\in C^{2,\alpha}(B_{1/2})$ with 
        \[\|u\|_{C^{2,\alpha}(B_{1/2})}\leq C.\]
    \end{theorem}
    
    Theorem \ref{THM: Savin thm} states that if a solution is sufficiently closed to some model solution (say a constant or a polynomial), then it must be regular. Theorems of this type can be traced back to the De Giorgi-Allard's $\varepsilon$-regularity theorem \cite{DG,Allard} in the minimal surface theory (see also \cite{Giusti, Simon-GMT}).

    For parabolic and nonlocal version of Theorem \ref{THM: Savin thm}, we refer to \cite{Wang, Yu-Hui}.

    The purpose of this note is generalizing Theorem \ref{THM: Savin thm} to non-homogeneous equations. Our main result is as follows:

    \begin{theorem}\label{theorem: main thm}
        For $0<\alpha<1$. Suppose that $F: \symmtrix \times \bb{R}^n\times \bb{R}\times B_1\to \bb{R}$ satisfies 
        
        $\mathrm{H}1)$ $F$ is elliptic and $\rho\,$-uniformly elliptic with ellipticity constants $0<\lambda\leq \Lambda$;
    
        $\mathrm{H}2')$ $F(0,0,0,x)\equiv 0$ and $F$ satisfies the structure condition: there exist constants $b_{0}, c_{0}>0$, such that for any $\|M\|,|p|,|q|, |z|,|s|\leq \rho$ and $x\in B_1$, 
        \begin{equation}\label{eq: structure condition}
            |F(M,p,z,x)-F(M,q,s,x)|\leq b_{0}|p-q|+c_{0}|z-s|.
        \end{equation}
    
        $\mathrm{H}3')$  $F\in C^1$ and $D_MF$ is uniformly continuous in a $\rho\,$-neighborhood of the set $\{(0,0,0,x): x\in B_1\}$ with modulus of continuity $\omega_{F}$.

        Let $u\in C(B_{1})$ be a viscosity solution to 
        \[F(D^2u, Du, u, x)=f(x)\qquad \mathrm{on}\ B_{1}.\]
        Then there exist constants $\delta,C>0$, depending only on $n,\alpha, \rho,\lambda,\Lambda$, $b_{0}$, $c_{0}$ and $\omega_{F}$, such that if
        \[ |F(M,p,z,x)-F(M,p,z,x')|\leq \delta |x-x'|^{\alpha}\quad \text{for all}\ \|M\|,|p|,|z|\leq \rho\ \text{and}\ x,x'\in B_1, \]
        and        
        \[ \|u\|_{L^{\infty}(B_{1})}\leq \delta, \qquad \|f\|_{C^{0,\alpha}(B_{1})}\leq \delta, \]
         then $u\in C^{2,\alpha}(B_{1/2})$ with 
        \[\|u\|_{C^{2,\alpha}(B_{1/2})}\leq C.\]
    \end{theorem}

   \begin{remark}
        When $F=F(M,x)$ is uniformly elliptic, Theorem \ref{theorem: main thm} was proved by dos Prazeres-Teixeira \cite{dos-Teixeira}. For general locally uniformly elliptic $F$, this result was first claimed by Lian-Zhang \cite{lian-zhang}. However, there is a mistake when proving weak Harnack in their original preprint and our proof is different from theirs. 
   \end{remark} 

    Let us explain the idea of proof step by step. We use the compactness argument: If the theorem were false, there would exist a sequence of small perturbation solutions with ``bad'' regularity.  The main task is to show that if we rescale these solutions, they will converge to a solution to the linearized operator of $F$ at the origin. Since this linearized operator is uniformly elliptic with constant coefficients, the limiting solution must be regular with good estimates. This leads to a contradiction.

    The most difficult step in the proof is to establish the convergence of the sequence mentioned above. To achieve this, we need to prove the weak Harnack inequality and H\"older estimates for locally uniformly elliptic equations, see Theorem \ref{THM: weak Harnack} and Theorem \ref{Thm: Holder regularity}.

    We employ the method of sliding paraboloids to prove the weak Harnack inequality. This method was first introduced by Cabr\'e \cite{Cabre-97} in his extension of Krylov-Safonov theory to manifolds with nonnegative sectional curvature. Later, Savin \cite{Savin-Thesis, Savin-Annals} generalized it and used it to solve the celebrated De Giorgi conjecture. This method has now become a standard approach to establish Harnack type estimates for non-uniformly elliptic equations, see \cite{Imbert-Silvestre}. It is worth mentioning that this method can also be employed to establish $W^{2,\varepsilon}$ estimates, we refer to \cite{Li-Li,Mooney-19, Baasandorj-Byun-Oh-24, Byun-Kim-Oh-25, Nascimento-Teixeira-25}.

    The advantage of this method is that it avoids the difficulties noted in Remark \ref{RMK: scaling cause problem}. We use paraboloids (or quadratic functions) to touch our solution. At the contact points, we get some useful information about the solution. Using the equation, we can also estimate the measure of the contact set. The key is that we do not rescale the solution and iterate as in the Caffarelli-Cabr\'{e}'s book \cite[Chapter 4]{CCbook}. Instead, we iterate by enlarging the opening of the touching paraboloids. This avoids the rescaling difficulties caused by nonuniform ellipticity.

    \subsection*{Applications}
    In the minimal surface theory, we can use the De Giorgi-Allard's $\varepsilon$-regularity theorem with blow-up analysis to prove partial regularity of minimal surfaces, which estimates the dimension of singular sets, see the book \cite{Giusti}. This gives us an idea that $\varepsilon$-regularity theorems can help establishing partial regularity results.

    For uniformly elliptic equations, Armstrong-Silvestre-Smart \cite{Armstrong-Silvestre-Smart} combined Lin's $W^{2,\varepsilon}$ estimates and Savin's $\varepsilon$-regularity theorem (Theorem \ref{THM: Savin thm}) to prove the following partial regularity result:

    \begin{theorem}[Armstrong-Silvestre-Smart]
        Let $F=F(M)$ be a uniformly elliptic operator with ellipticity constants $0<\lambda\leq\Lambda$ and satisfying $\mathrm{H}3')$, and let $u\in C(\Omega)$ be a viscosity solution to $F(D^2u)=0$ in a domain $\Omega\subset\bb{R}^n$. Then there is a universal constant $\varepsilon>0$ depending only on $n,\lambda$ and $\Lambda$, and a closed set $\Sigma\subset\Omega$ with Hausdorff dimension at most $n-\varepsilon$, such that $u\in C^{2,\alpha}(\Omega\setminus\Sigma)$ for every $0<\alpha<1$.
    \end{theorem}

    The idea of their proof is natural. Theorem \ref{THM: Savin thm} says that if a solution is sufficiently closed to some polynomial, then it must be regular locally. Now differentiating $F(D^2u)=0$ once and applying Lin's $W^{2,\varepsilon}$ estimates, we get a $W^{3,\varepsilon}$ estimate. This estimate tells us that at most points, $u$ can be locally approximated by quadratic polynomials and further have a measure estimate on the set of such points. Then the partial regularity follows. As an application of Theorem \ref{theorem: main thm}, the same partial regularity result still holds for non-homogeneous uniformly elliptic equations, which is proved by dos Prazeres-Teixeira \cite[Corollary 5.3]{dos-Teixeira}.
    
    More recently, in Shankar-Yuan's seminal work \cite{Shankar-Yuan-Ann-2025}, they proved a partial regularity result that any $2$-convex viscosity solution to $\sigma_{2}(D^2u)=1$ has singular set of Lebesgue measure zero. We now can generalize this result to the $\sigma_{k}\,$-Hessian equations $\sigma_{k}(D^2u)=f$ with Lipschitz positive right-hand side terms. 

    \begin{proposition}\label{PROP: sigma-k}
        For $2\leq k\leq n$, let $u$ be a k-convex viscosity solution to $\sigma_{k}(D^2u)=f$ on $B_{1}\subset\bb{R}^n$. Suppose that $f>0$ is Lipschitz.  Then

        (i) $u$ is twice differentiable almost everywhere, that is for $a.e.\,x_{0}\in B_{1}$, there exists a quadratic polynomial $Q_{x_{0}}$, such that
        \[|u(x)-Q_{x_{0}}(x)|=o(|x-x_{0}|^2).\]

        (ii) Denote the singular set of $u$ by $\Sigma:=\{x\in B_{1}: u\ \text{is not}\  C^{2,\alpha}\ \text{near}\ x\}.$ Then $\Sigma$ has Lebesgue measure zero.
    \end{proposition}

    This note is organized as follows. In Section \ref{Section: Weak Harnarck}, we derive the weak Harnack inequality and H\"older estimates for locally uniformly elliptic equations. In Sections \ref{Section: proof of main thm} and \ref{Section: Proof of Proposition}, we prove Theorem \ref{theorem: main thm} and Proposition \ref{PROP: sigma-k}.

    \bigskip
    \noindent
    {\bf Acknowledgments.} The author is supported by by National Key R\&D Program of China 2020YFA0712800. He is grateful to Professor Yu Yuan, Ravi Shankar and Kai Zhang for helpful discussions, and to Professors Qing Han and Yuguang Shi for their supports.

    \section{Weak Harnack Inequality and H\"older Estimates}\label{Section: Weak Harnarck}
    In this section, we establish the weak Harnack inequality for locally uniformly elliptic equations, then use it to derive the partial $C^{0,\alpha}$ estimates. We first introduce the definition of  Pucci's classes.
    \begin{definition}
        We say that $u\in \overline{S}_{\rho}(\lambda,\Lambda,b_{0},f)$, if for any $x_{0}\in B_{1}$ and $\varphi\in C^{2}(B_1)$ with
        \[ \varphi(x_{0})=u(x_{0}),\quad \varphi\leq u \ \mathrm{near}\ x_{0},\quad \text{and}\quad \|D^2\varphi(x_{0})\|, |D\varphi(x_{0})|, |\varphi(x_{0})|\leq \rho, \]
        there holds
        \[ \mathcal{M}^{-}_{\lambda,\Lambda}(D^2\varphi(x_{0}))-b_{0}|D\varphi(x_{0})|\leq f(x_{0}). \]

        Similarly, we say $u\in \underline{S}_{\rho}(\lambda,\Lambda,b_{0},f)$, if for any $x_{0}\in B_{1}$ and $\varphi\in C^{2}(B_1)$ with
        \[ \varphi(x_{0})=u(x_{0}),\quad \varphi\ge u \ \mathrm{near}\ x_{0},\quad \text{and}\quad \|D^2\varphi(x_{0})\|, |D\varphi(x_{0})|, |\varphi(x_{0})|\leq \rho, \]
        there holds
        \[ \mathcal{M}^{+}_{\lambda,\Lambda}(D^2\varphi(x_{0}))+b_{0}|D\varphi(x_{0})|\ge f(x_{0}). \]
        Here $\mathcal{M}^{\pm}=\mathcal{M}_{\lambda,\Lambda}^{\pm}$ are Pucci's extremal operators:
        \begin{align*}
            \mathcal{M}^{-}_{\lambda,\Lambda}(M)&=\lambda\sum_{e_{i}>0} e_{i}(M)+\Lambda\sum_{e_{i}<0} e_{i}(M),\\
            \mathcal{M}^{+}_{\lambda,\Lambda}(M)&=\Lambda\sum_{e_{i}>0} e_{i}(M)+\lambda\sum_{e_{i}<0} e_{i}(M).
        \end{align*}

        We also define
        \[ S_{\rho}^{*}(\lambda,\Lambda,b_{0},f)=\underline{S}_{\rho}(\lambda,\Lambda,b_{0},-|f|)\cap \overline{S}_{\rho}(\lambda,\Lambda,b_{0},|f|). \]

        For simplicity, if there is no ambiguity, we sometimes denote $\underline{S}_{\rho}(\lambda,\Lambda,b_{0},f), \overline{S}_{\rho}(\lambda,\Lambda,b_{0},f)$ and $S^{*}_{\rho}(\lambda,\Lambda,b_{0},f)$ by $\underline{S}_{\rho}(f), \overline{S}_{\rho}(f)$ and $S^{*}_{\rho}(f)$, respectively.
    \end{definition}
    The relation between functions in the Pucci's class and viscosity solutions to \eqref{main eq} is following:
    \begin{lemma}\label{LEMMA: relation between Pucci and solutions}
        Let $F$ be $2\rho\,$-uniformly elliptic  and satisfy the structure condition \eqref{eq: structure condition}.  Let $u$ be a viscosity supersolution (resp. subsolution) to 
        \[ F(D^2u, Du, u, x)=f(x) \qquad \mathrm{on}\ B_{1} \]
        with $\|u\|_{L^{\infty}(B_{1})}\leq \rho$. Then for any $\phi\in C^{2}(B_{1})$ with $\|\phi\|_{C^{1,1}(B_1)}\leq \rho$, we have
        \[  u-\phi\in \overline{S}_{\rho}\left( \lambda,\Lambda, b_{0}, \overline{f} \right) \ \left(\text{resp.}\ \in \underline{S}_{\rho}\left( \lambda,\Lambda, b_{0}, \underline{f} \right)  \right)   \]
        where
        \begin{align*}
            \overline{f}(x)&=f(x)+c_{0}|u(x)-\phi(x)|-F(D^2\phi(x), D\phi(x), \phi(x), x),\\
            \underline{f}(x)&= f(x)-c_{0}|u(x)-\phi(x)|-F(D^2\phi(x), D\phi(x), \phi(x), x).
        \end{align*}
    \end{lemma}

    \begin{proof}
        We only prove the Lemma for supersolution $u$. For any $\varphi\in C^{2}(B_{1})$, if $\varphi$ touches $u-\phi$ from below at $x_{0}\in B_{1}$ with $\|D^2\varphi(x_{0})\|, |D\varphi(x_{0})|, |\varphi(x_{0})|\leq \rho$. Then $\varphi+\phi$ touches $u$ from below at $x_{0}$. By the definition of supersolutions, $2\rho\,$-uniform ellipticity of $F$ and \eqref{eq: structure condition}, we have 
        \begin{align*}
            f(x_{0})&\ge F(D^2\varphi(x_{0})+D^2\phi(x_{0}),D\varphi(x_{0})+D\phi(x_{0}),\varphi(x_{0})+\phi(x_{0}), x_{0})\\
            &\ge \mathcal{M}^{-}(D^2\varphi(x_{0}))-b_{0}|D\varphi(x_{0})|-c_{0}|\varphi(x_{0})|+F(D^2\phi(x_{0}), D\phi(x_{0}),\phi(x_{0}),x_{0}).
        \end{align*}
        At the touch point $x_{0}$, we have $\varphi(x_{0})=u(x_{0})-\phi(x_{0})$. Therefore,
        \begin{align*}
            \mathcal{M}^{-}(D^2\varphi(x_{0}))-b_{0}|D\varphi(x_{0})|&\leq f(x_{0})+c_{0}|u(x_{0})-\varphi(x_{0})|- F(D^2\phi(x_{0}),D\phi(x_{0}),\phi(x_{0}),x_{0})\\
            &=\overline{f}(x_{0}).
        \end{align*}
        This verifies $u-\phi\in\overline{S}_{\rho}\left( \lambda,\Lambda, b_{0}, \overline{f} \right). $
    \end{proof}

    \subsection{Method of sliding paraboloids}
    For any $a>0$, $y\in \overline{B}_{1}$, we denote the quadratic polynomial 
    \[P_{a,y}(x)=-\dfrac{a}{2}|x-y|^2+\text{const}\]
    to be a concave paraboloid centered at $y$ of opening $a$.

    Given $V\subset \overline{B}_{1} $, for any $a>0$ and $y\in V$, we can slide the paraboloid $P_{a,y}$ from below until it touches $u$ by below at some point $x\in \overline{B}_{1}$. We note that such touch point must exist, since $P_{a,y}$ touches $u$ from below at $x$ if and only if $x$ is the minimum point of $u(z)+\frac{a}{2}|z-y|^2$ over $\overline{B}_{1}$. We collect all touch points, and denote the set of touch points by $T_{a}(V)$. Equivalently,
    \begin{align*}
        T_{a}(V)&=\left\{x\in\overline{B}_1: \exists\, y\in V, \mathrm{such\ that}\ u(x)+\dfrac{a}{2}|x-y|^2=\inf_{z}\left(u(z)+\dfrac{a}{2}|z-y|^2\right) \right\}.
    \end{align*}

    The first lemma states that if $u$ is a supersolution and the opening $a$ is small, then the measure of the set of touch points can control the measure of the set of centers.
    \begin{lemma}\label{measure lemma}
        Let $u\in \overline{S}_{\rho}(\lambda,\Lambda,b_{0},f)$ in $B_{1}$.
        Define $\Gamma=\frac{(n-1)\Lambda+2b_{0}+1}{\lambda}+1$. Assume that 
        
        (i) $\|u\|_{L^{\infty}(B_{1})}\leq \rho$ and $\|f\|_{L^{\infty}(B_{1})}<a\leq \rho/\Gamma;$
        
        (ii) $V\subset\overline{B}_{1}$ with $T_{a}(V)\subset B_{1}$.

Then we have 
        \[|T_{a}(V)|\ge \dfrac{1}{(1+\Gamma)^n}|V|. \]
    \end{lemma}
    \begin{proof}
        \emph{Step 1.} We first assume that $u$ is semi-concave. For general case, one can use Jensen's $\varepsilon$-envelope approximation to reduce the problem to the semi-concave case.

        By semi-concavity, the graph of $u$ admits at all points a touching paraboloid of opening $b$ from above. For any $x\in T_{a}(V)$, we can also touch the graph at $(x,u(x))$ by a paraboloid with vertex $y\in V$ and opening $a$ from below. Therefore, $u$ is differentiable at $x$. Moreover, $Du$ is Lipschitz on $T_{a}(V)$ with $[Du]_{Lip}\leq C(a,b)$.
        
        By the definition of $T_{a}(V)$, we know that $x$ is the minimum point of $u(z)+\frac{a}{2}|z-y|^2$. Since $T_{a}(V)\subset B_{1}$, we have $Du(x)+a(x-y)=0$. Thus the vertex $y$ is uniquely determined by $y=x+\frac{1}{a}Du(x).$
        
        Now, we define a Lipschitz map 
        \begin{align*}
            \mathcal{M}:T_{a}(V)&\longrightarrow V\\
            x\ &\longmapsto y:=x+\dfrac{1}{a}D u(x).
        \end{align*}
        By the definition of $T_{a}(V)$, $\mathcal{M}$ is surjective. Hence from the area forluma, we get
        \begin{align}\label{area formu}
            |V|\leq \int_{T_{a}(V)}|\det D\mathcal{M}(x)|\,\mathrm{d}x.
        \end{align}

    By the Alexandrov theorem, there exists $\mathcal{Z}\subset B_{1}$ with $|B_{1}\setminus \mathcal{Z}|=0$, such that $u$ is punctually twice differentiable on $\mathcal{Z}$. That is for any $z\in\mathcal{Z}$, 
    \[ u(x)=u(z)+Du(z)\cdot(x-z)+\dfrac{1}{2}(x-z)^{T}D^2u(z)(x-z)+o(|x-z|^2). \]
    Therefore, $D\mathcal{M}= I+\frac{1}{a}D^2u$ on $T_{a}(V)\cap\mathcal{Z}$.
    So far, we haven't used the equation. Next, we will use the equation to estimate $D^2u$ on $T_{a}(V)\cap \mathcal{Z}$.

    \begin{claim}
        $-aI\leq D^2u\leq \Gamma aI$ in $T_{a}(V) \cap \mathcal{Z}$.
    \end{claim}
    
    \noindent\emph{Proof of Claim.} The left inequality is obvious, since $u$ can be touched by a paraboloid of opening $a$ from below at any points in $T_{a}(V)$. 
    
    Suppose the right inequality fails, then there exist $x_{0}\in T_{a}(V)$ and a direction $e\in\bb{S}^{n-1}$, such that
        \[D^2u(x_{0})\ge \Gamma a e\otimes e-aI. \]
   For any $0<\varepsilon<a$ small, the function
   \[  \varphi_{\varepsilon}(x):= u(x_{0})- Du(x_{0})\cdot (x-x_{0})+ \dfrac{1}{2}(x-x_{0})^{T}(\Gamma ae\otimes e-aI)(x-x_{0})-\dfrac{\varepsilon}{2}|x-x_{0}|^2  \]
   will touch $u$ from below at $x_{0}$. Note that the eigenvalues of $D^2\varphi_{\varepsilon}(x_{0})= \Gamma a e\otimes e-(a+\varepsilon)I$ are $-a-\varepsilon,\cdots,-a-\varepsilon,(\Gamma-1)a-\varepsilon$. Since $\Gamma a\leq \rho$, we have $\|D^2\varphi_{\varepsilon}(x_{0})\|\leq \rho$. We can also check that $|D\varphi_{\varepsilon}(x_{0})|=|Du(x_{0})|\leq 2a\leq \rho$ and $ |\varphi_{\varepsilon}(x_{0})|=|u(x_{0})|\leq \rho$. From the definition of $\overline{S}_{\rho}(f)$, we obtain
        \begin{align*}
            f(x_{0})\ge \mathcal{M}^{-}(D^2\varphi_{\varepsilon}(x_{0}))-b_{0}|D\varphi_{\varepsilon}(x_{0})|\ge \lambda(\Gamma a-a-\varepsilon)-(n-a)\Lambda(a+\varepsilon)-2b_{0}a.
        \end{align*}
    Sending $\varepsilon\to 0$, we deduce that
    \[ f(x_{0})\ge [\lambda(\Gamma-1)-(n-1)\Lambda-2b_{0}] a\ge a> \|f\|_{L^{\infty}}. \]
    This leads a contradiction,  hence the claim holds.
    
    \vspace{0.3cm}
    
     Now from \eqref{area formu}, we conclude that
        \[|V|\leq \int_{T_{a}(V)\cap\mathcal{Z}}\left|\det \left(I+\dfrac{1}{a}D^2u(x)\right)\right|\,\mathrm{d}x\leq (1+\Gamma)^n|T_{a}(V)|.\]

    \vspace{0.3cm}

    \emph{Step 2. General case.} For general $u$, consider its Jensen's $\varepsilon$-envelope. For any $\varepsilon>0$, define
    \[ u_{\varepsilon}(x)=\inf_{y\in B_1} \left\{ u(y)+\dfrac{1}{\varepsilon}|y-x|^2 \right\}. \]
    By \cite[Chapter 5]{CCbook}, $u_{\varepsilon}$ is semi-concave and $u_{\varepsilon}$ converges to $u$ uniformly on compact sets of $B_{1}$. Moreover, we also have $u_{\varepsilon}\in \overline{S}_{\rho}(\lambda,\Lambda,b_{0},\|f\|_{L^{\infty}}).$ To see this, suppose $\varphi\in C^{2}$ touches $u_{\varepsilon}$ from below at $x_{0}$ with $\|D^2\varphi(x_{0})\|, |D\varphi(x_{0})|\leq \rho$. Let $\widetilde{x}_{0}\in B_{1}$ be the point such that
    \[ \inf_{y\in B_{1}}\left\{ u(y)+\dfrac{1}{\varepsilon}|y-x_{0}|^2 \right\}=u(\widetilde{x}_{0})+\dfrac{1}{\varepsilon}|\widetilde{x}_{0}-x_{0}|^2. \]
    Then $\widetilde{\varphi}(x)=\varphi(x+x_{0}-\widetilde{x}_{0})+u(\widetilde{x}_{0})-u(x_{0})$ touches $u$ from below at $\widetilde{x}_{0}$. It follows that
    \[ \mathcal{M}^{-}(D^2\varphi(x_{0}))-b_{0}|D\varphi(x_{0})|=\mathcal{M}^{-}(D^2\widetilde{\varphi}(\widetilde{x}_{0}))-b_{0}|D\widetilde{\varphi}(\widetilde{x}_{0})|\leq f(\widetilde{x}_{0})\leq  \|f\|_{L^{\infty}(B_{1})}. \]
    By Step 1, we conclude that
    \[ |T_{a,\varepsilon}(V)|\ge \dfrac{1}{(1+\Gamma)^n}|V|, \]
    where $T_{a,\varepsilon}$ are corresponding sets of touching points for $u_{\varepsilon}$. From the uniform convergence of $\{u_{\varepsilon}\}$, we can check that
    \[ \limsup_{k\to\infty}T_{a,1/k}(V)=\bigcap_{m=1}^{\infty}\bigcup_{k=m}^{\infty}T_{a,1/k}(V)\subset T_{a}(V). \]
    In conclusion, 
    \[ |T_{a}(V)|\ge \dfrac{1}{(1+\Gamma)^n}|V|. \]
    The proof is complete.
    \end{proof}

    For simplicity, we denote $T_{a}(\overline{B}_{1})$ by $T_{a}$.
    \begin{corollary}\label{density corollary}
        Let $u\in \overline{S}_{\rho}(\lambda,\Lambda,b_{0},f)$ and be non-negative in $B_{1}$. 
        Suppose $\rho\ge \rho_{0}:= 8\Gamma$. If
        \[ \|f\|_{L^{\infty}(B_{1})}\leq 8,\quad \|u\|_{L^{\infty}(B_{1})}\leq \rho,\quad \text{and}\quad  \inf_{B_{1/4}}u\leq 1. \]
        Then
        \begin{align*}
            \dfrac{|T_{8}\cap B_{1}|}{|B_{1}|}>\mu \quad and \quad \dfrac{|\{u\leq 2\}\cap B_{1}|}{|B_{1}|}>\mu,
        \end{align*}
        where $\mu=\mu(n,\lambda,\Lambda, b_{0})\in(0,1)$ is a universal constant.
    \end{corollary}
    \begin{proof}
        We claim that $T_{8}(B_{1/4})\subset B_{1}$, then by Lemma \ref{measure lemma}, we have
        \[|T_{8}\cap B_{1}|\ge |T_{8}(B_{1/4})|\ge\dfrac{1}{(1+\Gamma)^n}|B_{1/4}|.\]
        Note that $T_{8}(B_{1/4})\subset\{u\leq 2\}$ is obvious.
        
        It remains to prove the claim. For any $x\in T_{8}(B_{1/4})$, there exists $y\in B_{1/4}$, such that $x$ is the minimum point of $u(z)+4|z-x|^2$. We only need to rule out the possibility that it takes the minimum at $\partial B_{1}$.

        For $z\in \partial B_{1}$, since $u\ge0$, we have $u(z)+4|z-y|^2\ge 4\cdot (3/4)^2=9/4$. 
        
        On the other hand, since $\inf_{B_{1/4}}u\leq 1$, there exists $x_{1}\in B_{1/4}$, such that $u(x_{1})\leq 1$, then $u(x_{1})+4|x_{1}-y|^2\leq 2$. The claim easily holds.
    \end{proof}

    In the above argument, it is important to assume the touch point belongs to the interior of the ball, otherwise we cannot get the information of $\nabla u$ and $D^2 u$ at the touch point. 

    The next lemma is from a natural observation. If there is a touch point in the interior of the ball, then we enlarge the opening of the paraboloid and perturb the center slightly, the corresponding touch points also belong to the interior of the ball. Our proof is modified from \cite{Savin-07,Li-Li}.
    \begin{lemma}\label{lemma of enlarge opening}
        Let $u\in \overline{S}_{\rho}(\lambda,\Lambda,b_{0},f)$ in $B_{1}$. 
        There exists a universal $M=M(n,\lambda,\Lambda,b_{0})>1$, such that if 
        
        (i) $\|u\|_{L^{\infty}(B_1)}\leq \rho$ and $\|f\|_{L^{\infty}}< a\leq \frac{\rho}{M\Gamma}$;
        
        (ii) $B_{r}(x_{0})\subset B_{1}$ with $T_{a}\cap B_{r}(x_{0})\neq\varnothing$, \\
        then
        \[ \dfrac{|T_{Ma}\cap B_{r}(x_{0})|}{|B_{r}(x_{0})|}\ge \mu, \]
        where $\mu=\mu(n,\lambda,\Lambda,b_{0})\in(0,1)$ is a universal constant.
    \end{lemma}
    \begin{proof}
        Assume $x_{1}\in T_{a}\cap B_{r}(x_{0})$, then there exists a $y_{1}\in\overline{B_{1}}$, such that the paraboloid
        \[P_{a,y_{1}}(x)=-\dfrac{a}{2}|x-y_{1}|^2+u(x_{1})+\dfrac{a}{2}|x_{1}-y_{1}|^2\]
        touches $u$ from below at $x_{1}$. Next, we will enlarge the opening $a$ to $Ma$ and perturb the center $y_{1}$ slightly, we will show that the corresponding touch points also belong to $B_{r}(x_{0})$. We divide the proof into 3 steps.\\

        \emph{Step 1.} We claim that there exist $ x_{2}\in \overline{B_{r/2}}(x_{0})$ and $C_{0}=C_{0}(n,\lambda,\Lambda,b_{0}
        )>0$, such that
        \[ u(x_{2})-P_{a,y_{1}}(x_{2})\leq C_{0}ar^2. \]
        To see this, we consider the barrier function 
        \[\psi(x)=P_{a,y_{1}}(a)+ar^2\phi\left(\dfrac{|x-x_{0}|}{r}\right),\]
        where 
        \[\phi(t)=\begin{cases}
            \frac{1}{p}(t^{-p}-1) \quad &\mathrm{for}\ \frac{1}{2}\leq t\leq 1,\\
            \frac{1}{p}(2^{p}-1) \quad &\mathrm{for}\ t<\frac{1}{2}.
        \end{cases}\]
        Set $x_{2}$ be the minimum point of $u-\psi$ over $\overline{B_{r}}(x_{0})$.

       From the direct computation, for $x\in B_{r}(x_{0})\setminus \overline{B_{r/2}}(x_{0}) $, the eigenvalues of $D^2\psi(x)$ are 
        \[ a\left(-t^{-p-2}-1\right),\cdots, a\left(-t^{-p-2}-1\right), a((p+1)t^{-p-2}-1), \]
        where $t=\frac{|x-x_{0}|}{r}\in\left(\frac{1}{2},1\right)$.
        Moreover, $|D\psi(x)|\leq (2+rt^{-p-1})a$. Then we have
        \begin{align*}
            \mathcal{M}_{\lambda,\Lambda}^{-}(D^2\psi(x))-b_{0}|D\psi(x)|&=\lambda a((p+1)t^{-p-2}-1)-(n-1)\Lambda a\left(t^{-p-2}+1\right)\\
            &\quad -b_{0}(2+rt^{-p-1})a\\
            &\ge at^{-p-2}\left(\lambda p-2(n-1)\Lambda-4b_{0}\right)\\
            &\ge a>\|f\|_{L^{\infty}},
        \end{align*}
        provided $p\ge \frac{2(n-1)\Lambda+4b_{0}+1}{\lambda} $.
        We also require $M=M(p)$ sufficiently large, so that $\|D^2\psi(x)\|, |D\psi(x)| \leq Ma\leq \rho$. 
        Hence, by the definition of $\overline{S}_{\rho}(f)$, $u-\psi$ cannot attain its minimum in $B_{r}(x_{0})\setminus \overline{B_{r/2}}(x_{0})$, that is $x_{2}\notin B_{r}(x_{0})\setminus\overline{B_{r/2}}(x_{0})$.

        For $z\in\partial B_{r}(x_{0})$,  we have
        \begin{align*}
            u(z)-\psi(z)=u(z)-P_{a,y_{1}}(z)\ge0.
        \end{align*}
        However, 
        \begin{align*}
            u(x_{1})-\psi(x_{1})=u(x_{1})-P_{a,y_{1}}(x_{1})=-ar^2\phi\left(\dfrac{|x_{1}-x_{0}|}{r}\right)<0.
        \end{align*}
        Therefore $x_{2}\notin\partial B_{r}(x_{0})$.

        Finally, we have $x_{2}\in\overline{B_{r/2}}(x_{0})$ with $(u-\psi)(x_{2})<0$. Since $\phi\leq C_{0}$ for universal $C_{0}$, then
        \[u(x_{2})<\psi(x_{2})\leq P_{a,y_{1}}(x_{2})+C_{0}ar^2.\]
        \ \\
        \emph{Step 2.} We show that $T_{Ma}(V)\subset B_{r}(x_{0})$ with
        \[V=\overline{B_{r\frac{M-1}{8M}}}\left(\dfrac{1}{M}y_{1}+\dfrac{M-1}{M}x_{2}\right).\]

        For any $\widetilde{x}\in T_{Ma}(V)$, there exists $\widetilde{y}\in V$, such that the paraboloid
        \[ P_{Ma,\widetilde{y}}(x)=-\dfrac{Ma}{2}|x-\widetilde{y}|^2+u(\widetilde{x})+\dfrac{Ma}{2}|\widetilde{x}-\widetilde{y}|^2 \]
        touches $u$ from below at $\widetilde{x}$.

        Now we have two inequalities
        \begin{align*}
            \begin{cases}
                P_{a,y_{1}}(x)\leq u(x) \quad \mathrm{in}\ B_{1},\ \mathrm{with\ equality\ at}\ x_{1};  \\
                P_{Ma,\widetilde{y}}(x)\leq u(x) \quad \mathrm{in}\ B_{1},\ \mathrm{with\ equality\ at}\ \widetilde{x}.
            \end{cases}
        \end{align*}
        We first examine the difference of this two paraboloid, note that
        \begin{align*}
            P_{Ma,\widetilde{y}}(x)-P_{a,y_{1}}(x)&=-\dfrac{Ma}{2}|x-\widetilde{y}|^2+\dfrac{a}{2}|x-y_{1}|^2+\mathrm{Const}\\
            &=-\dfrac{(M-1)a}{2}|x-y^*|^2+R,
        \end{align*}
        where $y^*=\frac{M\widetilde{y}-y_{1}}{M-1}$ and $R$ denote the remainder constant term. Since $\widetilde{y}\in V$, by the definition of $V$, we have $y^*\in \overline{B_{r/8}}(x_{2})$.

        To estimate the remainder term $R$, we note that
        \[P_{a,y_{1}}(x)-\frac{(M-1)a}{2}|x-y^*|^2+R=P_{Ma,\widetilde{y}}(x)\leq u(x) \quad \mathrm{in}\ B_{1},\]
        then, at the point $x_{2}$, by Step 1, we get
        \begin{align*}
            R&\leq u(x_{2})-P_{a,y_{1}}(x_{2})+\dfrac{(M-1)a}{2}|x_{2}-y^*|^2\\
            &\leq C_{0}ar^2+\dfrac{M-1}{128}ar^2=\left(C_{0}+\dfrac{M-1}{128}\right)ar^2.
        \end{align*}

        Our goal is to show that $\widetilde{x}\in B_{r}(x_{0})$,  thus we need to estimate $|\widetilde{x}-x_{0}|$. To do this, we first estimate $|\widetilde{x}-y^*|$. Since
        \begin{align*}
            0&\leq u(\widetilde{x})-P_{a,y_{1}}(\widetilde{x})=P_{Ma,\widetilde{y}}(\widetilde{x})-P_{a,y_{1}}(\widetilde{x})\\
            &=-\dfrac{(M-1)a}{2}|\widetilde{x}-y^*|^2+R,
        \end{align*}
        then
        \[|\widetilde{x}-y^*|^2\leq \dfrac{2}{(M-1)a}R\leq \left(\dfrac{2C_{0}}{M-1}+\dfrac{1}{64}\right)r^2.\]
        Hence, $|\widetilde{x}-y^*|\leq \frac{r}{4}$ provided $M$ is universally large. Finally,
        \[|\widetilde{x}-x_{0}|\leq |\widetilde{x}-y^*|+|y^*-x_{2}|+|x_{2}-x_{0}|\leq \dfrac{r}{4}+\dfrac{r}{8}+\dfrac{r}{2}<r.\]
        \ \\
        \emph{Step 3.} Conclusion. By Lemma \ref{measure lemma}, we have
        \begin{align*}
            |T_{Ma}\cap B_{r}(x_{0}) |\ge |T_{Ma}(V)|\ge c|V|\ge c\left(r\dfrac{M-1}{8M}\right)^n=\widetilde{c}r^n,
        \end{align*}
        which implies
        \[ \dfrac{|T_{Ma}\cap B_{r}(x_{0})|}{|B_{r}(x_{0})|}>\mu \]
        for some universal $\mu\in(0,1)$.
    \end{proof}

    We also need the following covering lemma. For its proof, we refer to \cite[Lemma 2.1]{Imbert-Silvestre} and \cite[Lemma 2]{Li-Li}.
    \begin{lemma}\label{covering lemma}
        Let $E\subset F\subset B_{1}$, with $E\neq\varnothing$, and let $\mu\in(0,1)$. If for any ball $B\subset B_{1}$ with $B\cap E\neq \varnothing$,  we have $|B\cap F|>\mu|B|$, then
        \[|B_{1}\setminus F|\leq \left(1-\dfrac{\mu}{5^n}\right)|B_{1}\setminus E|.\]
    \end{lemma}

    By enlarging the opening of touching paraboloids and iteration, we have the following measure estimates for the set of touching points.
    \begin{corollary}\label{decay lemma}
        Let $u\in \overline{S}_{\rho}(\lambda,\Lambda,b_{0},f)$ and be non-negative in $B_{1}$.
        Suppose $\rho\ge \rho_{0}:= 8\Gamma$. If
        \[ \|f\|_{L^{\infty}(B_{1})}\leq 8,\quad \|u\|_{L^{\infty}(B_1)}\leq \rho,\quad\text{and}\quad \inf_{B_{1/4}}u\leq 1,   \]
         then
        \begin{align*}
            |B_{1}\setminus T_{8M^k}|\leq C_{n}(1-\theta)^k \quad \text{provided}\ 1\leq k\leq \dfrac{1}{\ln M}\ln\dfrac{\rho}{\rho_{0}}. 
        \end{align*}
        where $M$ is the constant in Lemma \ref{lemma of enlarge opening} and $\theta=\theta(n,\lambda,\Lambda,b_{0})\in(0,1)$.
    \end{corollary}
    \begin{proof}
        We prove this corollary by induction. For $k=1$, the result follows from Corollary \ref{density corollary}. Assume this result holds for $k-1$. Denote $E=T_{8M^{k-1}}\cap B_{1}$ and $F=T_{8M^{k}}\cap B_{1}$, then $E\subset F\subset B_{1}$. For any ball $B=B_{r}(x_{0})\subset B_{1}$ with $B\cap E\neq \varnothing$, that is $T_{8M^{k-1}}\cap B_{r}(x_{0})\neq \varnothing$. By Lemma \ref{lemma of enlarge opening}, we have
        \[\dfrac{|T_{8M^k}\cap B|}{|B|}\ge \mu \quad \mathrm{provided}\ 8M^k\leq \dfrac{\rho}{\Gamma},\ \mathrm{or\ equivalently}\ k\leq \dfrac{1}{\ln M}\ln\dfrac{\rho}{\rho_{0}}.\]
        Hence, by Lemma \ref{covering lemma}, we have $|B_{1}\setminus F|\leq (1-5^{-n}\mu)|B_{1}\setminus E|$, that is
        \[|B_{1}\setminus T_{8M^k}|\leq \left(1-\dfrac{\mu}{5^n}\right)|B_{1}\setminus T_{8M^{k-1}}|\leq (1-\theta)^{k}, \]
        provided $k\leq \dfrac{1}{\ln M}\ln\dfrac{\rho}{\rho_{0}}$ and if we take $\theta=5^{-n}\mu$.
    \end{proof}
    
    A direct consequence of Lemma \ref{decay lemma} is the following weak $L^{\varepsilon}$ estimate.

    \begin{corollary}[Weak $L^{\varepsilon}$ estimate]\label{Lepsilon est} 
    Under the hypothesis of Corollary \ref{decay lemma}, we have
    \[ |\{u>t\}\cap B_{1}|<Ct^{-\varepsilon} \quad \text{for}\ 0\leq t\leq 17\dfrac{\rho}{\rho_{0}}, \]
    where $\varepsilon=\varepsilon(n,\lambda,\Lambda,b_{0})>0$ is a universal constant.
    \end{corollary}

    \begin{proof}
        Since $\inf_{B_{1/4}}u\leq 1$, we have $T_{8M^k}\subset \{u\leq 1+16M^{k}\}\subset \{u\leq 17M^k\} $. For $17<t\leq 17\frac{\rho}{\rho_{0}}$, there exists $k\in \bb{N}$, such that $17M^k\leq t<17M^{k+1}$. Note that  $k\leq \frac{1}{\ln M}\ln\frac{\rho}{\rho_{0}}$. By Lemma \ref{decay lemma}, we get
        \begin{align*}
            |\{u>t\}\cap B_{1}|\leq |\{u>1+16M^{k}\}\cap B_{1}| \leq |B_{1}\setminus T_{8M^{k}}|\leq (1-\theta)^{k}\leq Ct^{-\varepsilon}.
        \end{align*}
        For $0<t\leq 17 $, the result clearly holds.
    \end{proof}

    The weak $L^{\varepsilon}$ estimate implies the following weak Harnack inequality.
    \begin{theorem}[Weak Harnack]\label{THM: weak Harnack}
        Let $u\in \overline{S}_{\rho}(\lambda,\Lambda,b_{0},f)$ and be non-negative in $B_{1}$.
        Suppose $\rho\ge\rho_{0}$. If
        \[ \|u\|_{L^{\infty}(B_1)}\leq \rho \quad  \text{and}\quad   \inf_{B_{1/4}}u+\dfrac{\|f\|_{L^{\infty}(B_{1})}}{8}\leq \frac{\rho}{\rho_{0}}. \]
        Then there exists $\varepsilon_{0}=\varepsilon_{0}(n,\lambda,\Lambda)>0$, such that
        \[ \|u_{\rho}\|_{L^{\varepsilon_{0}}(B_{1/4})}\leq C\left(\inf_{B_{1/4}}u+\|f\|_{L^{\infty}(B_{1})}\right), \]
        where $C=C(n,\lambda,\Lambda)>0$ and 
        \[u_{\rho}(x)=\begin{cases}
            u(x),\quad & \text{if} \ u(x)\leq 17\rho/\rho_{0},\\
            0          & \text{if} \  u(x)> 17\rho/\rho_{0}.
        \end{cases}\]
    \end{theorem}

    \begin{proof}
        Set $B:=\inf_{B_{1/4}}u+\frac{\|f\|_{L^{\infty}(B_{1})}}{8}\leq \frac{\rho}{\rho_{0}}$, and $v=u/B$, then $\inf_{B_{1/4}}v\leq 1$. Moreover, $v$ is non-negative and belongs to $\overline{S}_{\rho/B}\left(\lambda,\Lambda,b_{0},\frac{f}{B}\right)$. Since $\rho/B\ge \rho_{0}$, applying Corollary \ref{Lepsilon est} for $v$, we have
        \[|\{v>t\}\cap B_{1}|<Ct^{-\varepsilon},\quad \text{for}\ 0\leq t\leq 17\dfrac{\rho}{\rho_{0}B}.  \]
        Set $v_{\rho}=u_{\rho}/B$ and $\varepsilon_{0}=\varepsilon/2$, then we have
        \begin{align*}
            \int_{B_{1/4}}|v_{\rho}|^{\varepsilon_{0}}&=\varepsilon_{0}\int_{0}^{\infty}t^{\varepsilon_{0}-1}|\{v_{\rho}>t\}\cap B_{1/4}|\,\mathrm{d}t\\
            &\leq \varepsilon_{0}\left(\int_{0}^{1}|B_{1/4}|\,\mathrm{d}t+\int_{1}^{17\frac{\rho}{\rho_{0}B}}t^{\varepsilon_{0}-1}|\{v>t\}\cap B_{1}|\,\mathrm{d}t \right)\\
            &\leq \varepsilon_{0}|B_{1/4}|+C\int_{1}^{17\frac{\rho}{\rho_{0}B}}t^{\varepsilon_{0}-1-\varepsilon}\,\mathrm{d}t\\
            &\leq C
        \end{align*}

        Hence, $\|v_{\rho}\|_{L^{\varepsilon_{0}}(B_{1/4})}\leq C$, which means 
        \[\|u_{\rho}\|_{L^{\varepsilon_{0}}(B_{1/4})}\leq CB=C\left(\inf_{B_{1/4}}u+\|f\|_{L^{\infty}(B_{1})}\right).\]
    \end{proof}

    \subsection{H\"older estimates}\begin{theorem}[H\"older estimates]\label{Thm: Holder regularity}
        Let $u\in S^{*}_{\rho}(\lambda,\Lambda,b_{0},f)$ and be non-negative in $B_{1}$.
        Suppose $\rho>2\rho_{0}$. If
        \[ \|u\|_{L^{\infty}(B_{1})}\leq 1\qquad \text{and}\qquad \|f\|_{L^{\infty}(B_{1})}\leq\sigma, \]
        where $0<\sigma<1$ is a universal constant. Then
        \[ \mathrm{osc}_{B_{r}}u\leq Cr^{\alpha}, \quad \text{for}\ \sqrt{\dfrac{2\rho_{0}}{\rho}}\leq r\leq 1. \]
        where $C>0,\alpha\in(0,1)$ are universal constant depending only on $n,\lambda,\Lambda, b_{0}$.
    \end{theorem}

    \begin{proof}
        We only need to prove 
        \begin{align*}
            \mathrm{osc}_{B_{4^{-k}}}u\leq C(1-\kappa)^k, \quad \mathrm{for}\ 0\leq k\leq \dfrac{1}{2\ln 4}\ln\dfrac{\rho}{2\rho_{0}}.
        \end{align*}
        where $C,\kappa$ are universal. We prove it by induction. For $k=0$, it holds obviously. Assume it holds for $k$, set $r_{k}=4^{-k}$, $M_{k}=\sup_{B_{4^{-k}}}u, m_{k}=\inf_{B_{4^{-k}}}u $ and $\omega_{k}=M_{k}-m_{k} $.

        We assume that $\omega_{k}\ge r_{k}^2$, otherwise, we are done. Consider the rescaled function
        \[v(y)=\dfrac{u(r_{k}y)-m_{k}}{M_{k}-m_{k}},\quad y\in B_{1}, \]
        then $v$ is non-negative and belongs $S^{*}_{r_{k}^2\rho/\omega_{k}}(\lambda,\Lambda,b_{0},\widetilde{f})$  with $\widetilde{f}(\cdot)=\frac{r_{k}^2}{\omega_{k}}f(r_{k}\cdot)$. 

        Since $0\leq v\leq 1$ with $\mathrm{osc}_{B_{1}}v=1$, only one of following holds:
        \begin{align*}
            \dfrac{|\{v\ge1/2\}\cap B_{1}|}{|B_{1}|}\ge \dfrac{1}{2} \quad \text{or}\quad \dfrac{|\{v\leq 1/2\}\cap B_{1}|}{|B_{1}|}\ge \dfrac{1}{2}.
        \end{align*}
        Without loss of generality, assume the first case holds. Note that $\|u\|_{L^{\infty}(B_{1})}\leq 1$, then $\omega_{k}\leq 2$. For $k\leq \frac{1}{2\ln 4}\ln\frac{\rho}{2\rho_{0}}$ , we have 
        \[ \dfrac{r_{k}^2}{\omega_{k}}\rho\ge \dfrac{4^{-2k}}{2}\rho\ge\rho_{0}\quad \mathrm{and}\quad  \|\widetilde{f}\|_{L^{\infty}(B_{1})}\leq \dfrac{r_{k}^2}{\omega_{k}}\sigma\leq \sigma.\]
        By taking $\sigma<8$, we get
        \[\inf_{B_{1/4}}v+\dfrac{\|\widetilde{f}\|_{L^{\infty}(B_{1})}}{8}\leq 2\leq \dfrac{\rho}{\rho_{0}}.\]
        Therefore, we can applying Lemma \ref{Lepsilon est} to conclude that
        \begin{align*}
            c_{0}\leq \|v_{r_{k}^2\rho/2}\|_{L^{\varepsilon_{0}}(B_{1/4})}\leq C\left(\inf_{B_{1/4}} v+\|\widetilde{f}\|_{L^{\infty}(B_{1})}\right)\leq C\left(\inf_{B_{1/4}}v+\sigma\right).
        \end{align*}
        Choose $\sigma$ small enough such that $C\sigma\leq \frac{c_{0}}{2}$. Then $\inf_{B_{1/4}}v\ge \kappa$ for some universal $\kappa$, which means $\frac{m_{k+1}-m_{k}}{\omega_{k}}\ge \kappa$. Hence,
        \begin{align*}
            \omega_{k+1}=M_{k+1}-m_{k+1}\leq M_{k}-m_{k+1}\leq (1-\kappa)\omega_{k}\leq (1-\kappa)^{k+1}.
        \end{align*}
        By induction, the proof is complete.
    \end{proof}

    By translation, we have the following corollary which is useful in our future compactness argument.
    \begin{corollary}\label{coro: translated Holder reg}
        Let $u\in S^{*}_{\rho}(\lambda,\Lambda,b_{0},f)$ and be non-negative in $B_{1}$.
        Suppose $\rho>2\rho_{0}$. If
         \[  \|u\|_{L^{\infty}(B_{1})}\leq 1\qquad \text{and}\qquad \|f\|_{L^{\infty}(B_{1})}\leq\sigma, \]
        where $0<\sigma<1$ is a universal constant. Then
        \[ \mathrm{osc}_{B_{r}(x_{0})}u\leq Cr^{\alpha}, \quad \text{for any}\ x_{0}\in B_{1/2}\ \text{and}\ \sqrt{\dfrac{2\rho_{0}}{\rho}}\leq r\leq \dfrac{1}{2}. \]
        where $C>0,\alpha\in(0,1)$ are universal constants depending only on $n,\lambda,\Lambda,b_{0}$. 

        Consequently, 
        \[|u(x_{1})-u(x_{2})|\leq C|x_{1}-x_{2}|^{\alpha},\quad \text{for any}\ x_{1},x_{2}\in B_{1/2}\ \text{and}\ \sqrt{\dfrac{2\rho_{0}}{\rho}}\leq |x_{1}-x_{2}|\leq \dfrac{1}{2}. \]
    \end{corollary}

\section{Proof of Theorem \ref{theorem: main thm}}\label{Section: proof of main thm}
In this section, we prove Theorem \ref{theorem: main thm}.  The key step in the proof is the improvement of flatness Lemma \ref{lemma: improvement of flatness}. Once establishing it, the Caffarelli's iteration argument can go through.

We first introduce some notations. Let $P(x)=a+b\cdot x+ x^{T}Cx$ be a quadratic polynomial, where $a\in\bb{R}, b\in\bb{R}^n$ and $C\in\symmtrix$. 
Define
\[\|P\|_{r}:=|a|+r|b|+r^2\|C\|\ \text{for}\ r>0,\quad \text{and}\quad \|P\|=\|P\|_{1}.\]

\begin{lemma}[Improvement of flatness] \label{lemma: improvement of flatness}
    For $0<\alpha<1$. Let $F$ satisfy $\mathrm{H}1), \mathrm{H}2')$, $\mathrm{H}3')$, and let $u\in C(B_1)$ be a viscosity solution to
    \[F(D^2u, Du, u, x)=f \quad \mathrm{on}\ B_{1}.\]
    There exist universal constants $r_{0},\delta_{0}, \eta\in(0,1)$ and $C>0$, which depend only on $n,\alpha, \lambda,\Lambda,\rho, b_{0},c_{0}$ and $\omega_{F}$, such that if 
    \[ \|F(M,p,z,x)-F(M,p,z,0)\|_{L^{\infty}(B_{r})} \leq \delta_{0}r^{\alpha}\quad \text{for any}\ \|M\|,|p|, |z| \leq \rho, \]
    and
    \[\|u-P_{0}\|_{L^{\infty}(B_{r})}\leq r^{2+\alpha},\qquad \|f-f(0)\|_{L^{\infty}(B_{r})}\leq \delta_{0}r^{\alpha}. \]
    for some $r\leq r_{0}$ and some quadratic polynomial $P_{0}$ with
    \[ \|P_{0}\|\leq Cr_{0}^{\alpha},\qquad and \qquad F(D^2P_{0}, DP(0), P(0), 0)=f(0). \]
    
     Then there exists a quadratic polynomial $P$, such that 
    \begin{align*}
        \|u-P\|_{L^{\infty}(B_{\eta r})}\leq (\eta r)^{2+\alpha},
    \end{align*}
    and 
    \[ \|P-P_{0}\|_{\eta r}\leq C(\eta r)^{2+\alpha}, \qquad F(D^2P, DP(0), P(0), 0)=f(0). \]
    
    \end{lemma}

    \begin{proof}
        We prove this lemma by contradiction. If this lemma fails, for universal $\eta\in(0,1)$ and $C>0$ to be chosen later, there exist sequences $\{F_{k}\}$, $\{u_{k}\}$, $\{f_{k}\}$, $\{P_{k}\}$ and $\{r_{k}\}$, such that

       i) $\{F_{k}\}$ satisfy $\mathrm{H}1), \mathrm{H}2'),\mathrm{H3')}$ and $F_{k}(D^2u_{k}, Du_{k}, u_{k}, x)=f_{k}$ on $B_{1}$ in the viscosity sense;

       ii) $r_{k}\leq 1/k$, $\|u_{k}-P_{k}\|_{L^{\infty}(B_{r_{k}})}\leq r_{k}^{2+\alpha}$ and $\|f_{k}-f_{k}(0)\|_{L^{\infty}(B_{r_{k}})}\leq r_{k}^{\alpha}/k$;

       iii) For any $\|M\|,|p|, |z|\leq \rho$, there holds $\|F_{k}(M,p,z,x)-F_{k}(M,p,z,0)\|_{L^{\infty}(B_{r_{k}})}\leq r_{k}^{\alpha}/k$;
       
       iv)  $\|P_{k}\|\leq C/k^{\alpha}$ with $F_{k}(D^2P_{k}, DP_{k}(0), P_{k}(0),0)=f_{k}(0)$;

       v) However, for any $k$, and any quadratic polynomial $P$ with $F_{k}(D^2P, DP(0), P(0),0)=f_{k}(0)$ and $\|P-P_{k}\|_{\eta r_{k}}\leq C(\eta r_{k})^{2+\alpha}$, we have $\|u_{k}-P\|_{L^{\infty}(B_{\eta r_{k}})}>(\eta r_{k})^{2+\alpha}$.\ \\
       
       \emph{Step 1.} Passing to limit. 
       
       Since $\{F_{k}\}$ are $\rho$-uniformly elliptic, and $D_{M}F$ has a uniform modulus of continuity $\omega_{F}$ near the origin. By Arzela-Ascoli theorem, up to subsequence, we deduce that $D_{M}F_{k}(0,0,0,0)$ converges to some matrix $A=(a^{ij})$ with $\lambda I\leq A\leq \Lambda I$.

       Consider the rescaled function 
       \begin{align}\label{eq: def of rescaled function vk}
           v_{k}(x)=\dfrac{u_{k}(r_{k}x)-P_{k}(r_{k}x)}{r_{k}^{2+\alpha}},\qquad \text{for}\ x\in B_{1}.
       \end{align}
       then $v_{k}$ is the viscosity solution to 
       \[ \widetilde{F}_{k}(D^2v_{k}, Dv_{k}, v_{k}, x)=\widetilde{f}_{k}(x)\qquad \mathrm{on}\ B_{1}, \]
       where
       \[\widetilde{F}_{k}(M,p,z,x)=\dfrac{1}{r_{k}^{\alpha}}F_{k}(r_{k}^{\alpha}M+D^2P_{k}, r_{k}^{1+\alpha}p+DP_{k}(r_{k}x)+ r^{2+\alpha}z+P_{k}(r_{k}x), r_{k}x)    \]
       and
       \[   \widetilde{f}_{k}(x)=\dfrac{1}{r_{k}^{\alpha}}f_{k}(r_{k}x).     \]
       Since each $F_{k}$ is $\rho$-uniformly elliptic, and $\|P_{k}\|\to0$, we know that for sufficiently large $k$,  $\widetilde{F}_{k}$ is $\frac{1}{2}r_{k}^{-\alpha}\rho$-uniformly elliptic with same ellipticity constants.
       Furthermore, we also have $\|v_{k}\|_{L^{\infty}(B_{1})}\leq 1\leq \frac{1}{4}r_{k}^{-\alpha}\rho$. Applying Lemma \ref{LEMMA: relation between Pucci and solutions} with $\phi=0$, we obtain 
       \[ v_{k}\in S^{*}_{\frac{1}{4}r_{k}^{-\alpha}\rho}(\lambda, \Lambda, b_{0}, f_{k}^{*}), \]
       where
       \begin{align*}
           f_{k}^{*}(x)&=|\widetilde{f}_{k}-\widetilde{F}_{k}(0,0,0,x)|+r_{k}^2c_{0}|v_{k}(x)|\\
           &=\dfrac{|f_{k}(r_{k}x)-F_{k}(D^2P_{k},DP_{k}(r_{k}x), P(r_{k}x), r_{k}x)|}{r_{k}^{\alpha}}+r_{k}^2c_{0}|v_{k}(x)|.
       \end{align*}
        Notice that $F_{k}(D^P_{k}, DP_{k}(0), P_{k}(0), 0)=f_{k}(0)$, we have
        \begin{align*}
            &|f_{k}(r_{k}x)-F_{k}(D^2P_{k},DP_{k}(r_{k}x), P(r_{k}x), r_{k}x)|\\
            &\leq |f_{k}(r_{k}x)-f_{k}(0)|+|F_{k}(D^2P_{k},DP_{k}(r_{k}x), P(r_{k}x), r_{k}x)- F_{k}(D^2P_{k},DP_{k}(r_{k}x), P(r_{k}x), 0)|\\
            &\ \quad +|F_{k}(D^2P_{k},DP_{k}(r_{k}x), P(r_{k}x), 0)- F_{k}(D^2P_{k}, DP_{k}(0), P_{k}(0), 0)|\\
            &\ \leq \dfrac{r_{k}^{\alpha}}{k}+\dfrac{r_{k}^{\alpha}}{k}+b_{0}|DP_{k}(r_{k}x)-DP_{k}(0)|+ c_{0}|P_{k}(r_{k}x)-P_{k}(0)|=2\dfrac{r_{k}^{\alpha}}{k}+ O(r_{k}).
        \end{align*}
        Let $\rho_{0}$ and $\sigma$ be the constants in Corollary \ref{coro: translated Holder reg}, for $k$ sufficiently large, we can ensure that
        \[ \dfrac{1}{4}r_{k}^{-\alpha}\rho\ge 2\rho_{0}\qquad \text{and}\qquad \|f_{k}^{*}\|_{C^{0,\alpha}(B_{1})}\leq \frac{2}{k}+O(r_{k}^{1-\alpha})\leq \sigma.  \]
        Applying Corollary \ref{coro: translated Holder reg} for $v_{k}$, we have for some universal $\alpha_{0}\in (0,1)$,
       \begin{align*}
           |v_{k}(x_{1})-v_{k}(x_{2})|\leq C|x_{1}-x_{2}|^{\alpha_{0}}, \quad \text{for any}\ x_{1},x_{2}\in B_{1/2}\ \text{and}\ \sqrt{\dfrac{8\rho_{0}}{r_{k}^{-\alpha}\rho}}\leq |x_{1}-x_{2}|\leq \dfrac{1}{2}.
       \end{align*}

       Set $d_{k}= \sqrt{\frac{8\rho_{0}}{r_{k}^{-\alpha}\rho}} $. Since $r_{k}\to 0$, for any $\varepsilon>0$, there exists $k_{0}\in\bb{N}^{+}$, such that for any $k\ge k_{0}$, we have $2Cd_{k}^{\alpha}<\varepsilon$ . Now for any $x_{1},x_{2}\in B_{1/2}$, with $|x_{1}-x_{2}|\leq d_{k_{0}}$, we can choose $x_{3}\in B_{1/2}$ such that $|x_{1}-x_{3}|=|x_{2}-x_{3}|=d_{k_{0}}$. Now for any $k\ge k_{0}$, we have
       \[|v_{k}(x_{1})-v_{k}(x_{2})|\leq |v_{k}(x_{1})-v_{k}(x_{3})|+|v_{k}(x_{2})-v_{k}(x_{3})|\leq 2Cd_{k_{0}}^{\alpha}<\varepsilon. \]
       Therefore, $\{v_{k}\}$ is equicontinuous. By Arzela-Ascoli theorem, up to subsequence, $v_{k}\to v_{0}$ in $C(B_{1/2})$.\ \\

       \emph{Step 2.} Show that $v_{0}$ is the viscosity solution to $a^{ij}D^{2}_{ij}v_{0}=0$ on $B_{1/2}$.  

       For any $x_{0}\in B_{1/2}$ and $\varphi\in C^2$ such that touches $v_{0}$ from above at $x_{0}$. From the uniform convergence of $\{v_{k}\}$, for sufficiently large $k$, there exist sequences $c_{k}\to 0$ and $x_{k}\to x_{0}$, such that $\varphi+c_{k}$ touches $v_{k}$ from above at $x_{k}$.
       Recall the relation \eqref{eq: def of rescaled function vk} between $v_{k}$ and $u_{k}$ , we know that $\widetilde{\varphi}_{k}(x):= r_{k}^{2+\alpha}(\varphi(r_{k}^{-1}x)+c_{k})+P_{k}(x)$ touches $u_{k}$ by above at $\widetilde{x}_{k}:= r_{k}x_{k}$. Now using the equation of $u_{k}$, we obtain 
       \begin{align*}
           f_{k}(\widetilde{x}_{k})&\leq F_{k}(r_{k}^{\alpha} D^{2}\widetilde{\varphi}+D^2P_{k}, r_{k}^{1+\alpha}D\widetilde{\varphi}+DP_{k}, r_{k}^{2+\alpha}\widetilde{\varphi}+P_{k}, \widetilde{x}_{k} ).
       \end{align*}
       To estimate the right-hand side, for sufficiently large $k$, we rewrite it by
       \begin{align*}
           \mathrm{RHS}&= F_{k}(r_{k}^{\alpha} D^{2}\widetilde{\varphi}+D^2P_{k}, r_{k}^{1+\alpha}D\widetilde{\varphi}+DP_{k}, r_{k}^{2+\alpha}\widetilde{\varphi}+P_{k}, \widetilde{x}_{k} )\\
           &\quad -F_{k}(r_{k}^{\alpha} D^{2}\widetilde{\varphi}+D^2P_{k}, r_{k}^{1+\alpha}D\widetilde{\varphi}+DP_{k}, r_{k}^{2+\alpha}\widetilde{\varphi}+P_{k}, 0 )\\
           &\quad +F_{k}(r_{k}^{\alpha} D^{2}\widetilde{\varphi}+D^2P_{k}, r_{k}^{1+\alpha}D\widetilde{\varphi}+DP_{k}, r_{k}^{2+\alpha}\widetilde{\varphi}+P_{k}, 0 )\\
           &\quad -F_{k}(r_{k}^{\alpha} D^{2}\widetilde{\varphi}+D^2P_{k}, DP_{k}(0), P_{k}(0), 0 )\\
           &\quad +F_{k}(r_{k}^{\alpha} D^{2}\widetilde{\varphi}+D^2P_{k}, DP_{k}(0), P_{k}(0), 0 )\\
           &\quad -F_{k}(D^2P_{k}, DP_{k}(0), P_{k}(0), 0 ) + f_{k}(0)\\
           &\leq  \dfrac{r_{k}^{\alpha}}{k}+b_{0}r_{k}^{1+\alpha}\|D\varphi\|_{L^{\infty}} +b_{0}|DP_{k}(\widetilde{x}_{k})-DP_{k}(0)|+ c_{0}r_{k}^{2+\alpha}\|\varphi\|_{L^{\infty}}+ c_{0}|P_{k}(\widetilde{x}_{k})-P_{k}(0)|\\
           &\quad +F^{ij}_{k}(\theta r_{k}^{\alpha} D^2\widetilde{\varphi}+ D^2P_{k}, DP_{k}(0), P_{k}(0), 0) D^2_{ij}\widetilde{\varphi}(\widetilde{x}_{k})+f_{k}(0)\\
           &\leq \dfrac{r_{k}^{\alpha}}{k}+O(r_{k})+ r_{k}^{\alpha} F^{ij}_{k}(\theta r_{k}^{\alpha} D^2\widetilde{\varphi}+ D^2P_{k}, DP_{k}(0), P_{k}(0), 0) D^2_{ij}\varphi(x_{k})+ f_{k}(0).
       \end{align*}
       Here $F_{k}^{ij}(M,p,z,x)=\frac{\partial F_{k}}{\partial m_{ij}}(M,p,z,x)$, and $\theta\in(0,1)$ is obtained from the mean value theorem. Since $D_{M}F$ is uniformly  continuous with modulus $\omega_{F}$, we conclude that
       \begin{equation}
           \begin{split}
               0&\leq r_{k}^{\alpha}F^{ij}_{k}(0,0,0,0)D^{2}_{ij}\varphi(x_{k})+Cr_{k}^{\alpha}\omega_{F}(Cr_{k}^{\alpha})+ \dfrac{r_{k}^{\alpha}}{k}+ O(r_{k})+
               |f_{k}(r_{k}x_{k})-f_{k}(0)|\\
               &= r_{k}^{\alpha}( F_{k}^{ij}(0,0,0,0)D^2_{ij}\varphi(x_{k})+o(1) ).
           \end{split}
       \end{equation}
       
       Therefore, $F_{k}^{ij}(0,0,0,0)D^2_{ij}\varphi(x_{k})+o(1)\ge 0$, sending $k\to\infty$, we get $a^{ij}D^2_{ij}\varphi(x_{0})\ge0$. This implies $v_{0}$ is a viscosity subsolution to $a^{ij}D^2_{ij}v_{0}=0$. Similarly, we can also prove that $v_{0}$ is a supersolution. \ \\ 

       \emph{Step 3.} Derive the contradiction. 

      Now, $v_{0}$ is the viscosity solution to a linear uniformly elliptic equation with constant coefficients. Moreover, $\|v_{0}\|_{L^{\infty}(B_{1/2})}\leq 1$. From the classical elliptic theory, we have $v_{0}\in C^{\infty}(B_{1/2})$, with $\|D^{m}v_{0}\|_{L^{\infty}(B_{1/4})}\leq C_{m}=C_{m}(n,\lambda,\Lambda).$ Let $P$ be the quadratic Taylor polynomial of $v_{0}$, that is $P(x)=v_{0}(0)+D v_{0}(0)\cdot x+\frac{1}{2}x^{T}D^2v_{0}(0)x$, then 
       \[ \|P\|\leq C_{2}\qquad \text{and} \qquad \|v_{0}-P\|_{L^{\infty}(B_{\eta})}\leq C_{3}\eta^3.\]
       By taking $\eta$ small and $C$ large such that $C_{3}\eta^{1-\alpha}\leq 1/2$ and $C_{2}\leq (C-1)\eta^{2+\alpha}$, then we have
       \begin{align}\label{eq: estimate on P and v0-P}
           \|P\|\leq (C-1)\eta^{2+\alpha}\qquad \text{and} \qquad \|v_{0}-P\|_{L^{\infty}(B_{\eta})}\leq \dfrac{1}{2}\eta^{2+\alpha}.
       \end{align}

       Similar as before, using the uniform continuity of $D_{M}F$ again, we have
       \begin{align*}
           F_{k}(r_{k}^{\alpha}D^2P+D^2P_{k}(0), r_{k}^{1+\alpha} DP(0)+DP_{k}(0), r_{k}^{2+\alpha}P(0)+P_{k}(0), 0)=f_{k}(0)+o(r_{k}^{\alpha}).
      \end{align*}
      By the $\rho\,$-uniform ellipticity of $F$, for $k$ large, there exists $a_{k}=o(1)$, such that
      \begin{align*}
           F_{k}(r_{k}^{\alpha}(D^2P+a_{k}I)+D^2P_{k}(0), r_{k}^{1+\alpha} DP(0)+DP_{k}(0), r_{k}^{2+\alpha}P(0)+P_{k}(0), 0)=f_{k}(0).
      \end{align*}
      Now, we define
      \[ Q_{k}(x)=P_{k}(x)+r_{k}^{2+\alpha}\left(P(r_{k}^{-1}x)+\dfrac{a_{k}}{2r_{k}^2}|x|^2\right). \]
      Then $F_{k}(D^2Q_{k}, DQ_{k}(0), Q_{k}(0),0)=f_{k}(0).$
      By \eqref{eq: estimate on P and v0-P},  we also have
      \[\|Q_{k}-P_{k}\|_{\eta r_{k}}\leq r_{k}^{2+\alpha}(\|P\|+o(1))\leq C(\eta r_{k})^{2+\alpha}\]
      for sufficiently large $k$.
    However, by \eqref{eq: estimate on P and v0-P}, for sufficiently large $k$, there holds
      \begin{align*}
          \sup_{B_{\eta r_{k}}}|u_{k}-Q_{k}|&=r_{k}^{2+\alpha}\sup_{B_{\eta}}\left|v_{k}-P-\dfrac{a_{k}}{2}|x|^2\right| \\
          &=r_{k}^{2+\alpha}\left(\sup_{B_{\eta}}|v_{k}-v_{0}|+\sup_{B_{\eta}}|v_{0}-P|+o(1)\eta^2\right)\\
          &\leq r_{k}^{2+\alpha}\left(o(1)+\dfrac{1}{2}\eta^{2+\alpha}+o(1)\eta^2\right)\\
          &\leq (\eta r_{k})^{2+\alpha}.
      \end{align*}
    This contradicts (v) in the assumptions before step 1. Now the proof is complete.
    \end{proof}

    Finally, applying Caffarelli's iteration argument, we establish the interior pointwise $C^{2,\alpha}$ estimates for flat solutions, it implies Theorem \ref{theorem: main thm}.

    \begin{lemma}
       For $0<\alpha<1$. Let $F$ satisfy $\mathrm{H}1), \mathrm{H}2'),\mathrm{H}3')$, and let $u\in C(B_{1})$ be a viscosity solution to 
        \[F(D^2u, Du, u, x)=f\qquad \mathrm{on}\ B_{1}.\]
        Then there exist constants $r_{0},\delta,C>0$, depending only on $n,\alpha, \rho,\lambda,\Lambda, b_{0}, c_{0}$ and $\omega_{F}$, such that if
        \[ |F(M,p,z,x)-F(M,p,z,x')|\leq \delta |x-x'|^{\alpha}\quad \text{for all}\ \|M\|,|p|,|z|\leq \rho\ \text{and}\ x,x'\in B_1, \]
        and        
        \[ \|u\|_{L^{\infty}(B_{1})}\leq \delta, \qquad \|f\|_{C^{0,\alpha}(B_{1})}\leq \delta. \]
        Then there exists a quadratic polynomial $P$ with $F(D^2P, DP(0), P(0), 0)=f(0)$ and $\|P\|\leq C\delta$, such that 
        \[|u(x)-P(x)|\leq C\delta |x|^{2+\alpha},\qquad  \text{for any} \ x\in B_{r_{0}}.\]
    \end{lemma}
    \begin{proof}
        \emph{Step 1.} We first claim that there exists a sequence of quadratic polynomials $\{P_{k}\}_{k=0}^{\infty}$, such that
        \begin{align}\label{eq: u-Pk}
            \|u-P_{k}\|_{L^{\infty}(B_{\eta^k r_{0}})}\leq (\eta^k r_{0})^{2+\alpha},
        \end{align}
        and
        \begin{align}\label{eq: condition on Pk}
            F(D^2P_{k}, DP_{k}(0), P_{k}(0), 0)=f(0),\qquad \|P_{k}-P_{k-1}\|_{\eta^k r_{0}}\leq C(\eta^k r_{0})^{2+\alpha},
        \end{align}
        where $\eta, r_{0}$ and $C$ are universal constants in Lemma \ref{lemma: improvement of flatness}.

        We prove it by induction. Set $P_{-1}\equiv 0$. For $k=0$, since $F(0,0,0,x)=0$ and $F$ is uniformly elliptic in the neighborhood of $(0,0,0,x)$, for $\delta$ sufficiently small, there exists $t\in \bb{R}$, such that $F(tI,0,0,0)=f(0)$ with $|t|\leq |f(0)|/n\lambda\leq \delta/n\lambda$. Take $P_{0}=t|x|^2/2$, then
        \[ \|P_{0}\|_{r_{0}}\leq |t|r_{0}^2\leq Cr_{0}^{2+\alpha},\qquad  \|u-P_{0}\|_{L^{\infty}(B_{r_{0}})}\leq c(n,\lambda)\delta\leq r_{0}^{2+\alpha}.\]
        By taking $\delta=c^{-1}r_{0}^{2+\alpha}$,  \eqref{eq: u-Pk} and \eqref{eq: condition on Pk} hold for $k=0$.

        Suppose that the claim holds for $k=k_{0}$, for any $j\leq k_{0}$, \eqref{eq: condition on Pk}  implies $\|P_{j}-P_{j-1}\|\leq C(\eta^{j}r_{0})^{\alpha}$, then we have
        \[ \|P_{k_{0}}\|\leq\sum_{j=1}^{k_{0}}\|P_{j}-P_{j-1}\|+\|P_{0}\|\leq Cr_{0}^{\alpha}. \]
        Now we can applying Lemma \ref{lemma: improvement of flatness} for $r=\eta^{k_{0}}r_{0}$ to get a $P_{k_{0}+1}$ satisfying \eqref{eq: u-Pk} and \eqref{eq: condition on Pk}. By induction, the claim holds.

        \emph{Step 2.} Show that  $\{P_{k}\}$ will converges, and its limit is as desired.

        Write $P_{k}(x)-P_{k-1}(x)=a_{k}+b_{k}\cdot x+ x^{T}C_{k}x $. By \eqref{eq: condition on Pk}, we know that
        \[ |a_{k}|+(\eta^{k}r_{0})|b_{k}|+(\eta^{k}r_{0})^2\|C_{k}\|\leq C(\eta^kr_{0})^{2+\alpha}. \]
        Thus, $|a_{k}|\leq C(\eta^kr_{0})^{2+\alpha}, |b_{k}|\leq C(\eta^kr_{0})^{1+\alpha}$ and $\|C_{k}\|\leq C(\eta^kr_{0})^{\alpha}$. This implies $\sum_{k}(P_{k}-P_{k-1})$ converges, denote the limit quadratic polynomial by $P$, we have 
        \begin{align*}
            F(D^2P, DP(0), P(0), 0)=f(0) \quad \text{and}\quad P(x)=\sum_{k=0}^{\infty}a_{k}+\left(\sum_{k=0}^{\infty}b_{k}\right)\cdot x+x^{T}\left(\sum_{k=0}^{\infty}C_{k}\right)x.
        \end{align*}
        It is clear that
        \begin{align}\label{eq: estimate for P}
            \|P\|\leq C\delta\qquad \text{and} \qquad \sup_{B_{\eta^k r_{0}}}|P-P_{k}|\leq C\delta\eta^{(2+\alpha)k}.
        \end{align}

        Finally, for any $x\in B_{r_{0}}$, there exists an integer $k\ge0$, such that $\eta^{k+1}r_{0}< |x|\leq \eta^{k}r_{0}$. From \eqref{eq: u-Pk} and \eqref{eq: estimate for P}, we finally conclude that
        \[|u(x)-P(x)|\leq |u(x)-P_{k}(x)|+|P(x)P_{k}(x)|\leq C\delta\eta^{k(2+\alpha)}\leq C\delta|x|^{2+\alpha}.\]
        The proof is complete.
    \end{proof}
        
\section{Proof of Proposition \ref{PROP: sigma-k}} \label{Section: Proof of Proposition}
    In this section, we prove the partial regularity, Proposition \ref{PROP: sigma-k}, for the $\sigma_{k}\,$-Hessian equation. 

    \begin{definition}
        We say a $C^2$ function $u$ is $k$-convex, if its eigenvalues of Hessian satisfy
        \begin{equation}\label{eq: k-convexity}
            \lambda(D^2u)\in \Gamma_{k}:=\{\lambda\in \bb{R}^n: \sigma_{i}(\lambda)>0,\ \text{for}\ i=1,\cdots,k\}.
        \end{equation}
        We can also define the $k$-convexity for continuous functions if \eqref{eq: k-convexity} holds in the viscosity sense.
    \end{definition}
    
    We first consider the almost everywhere twice differentiability assertion in Proposition \ref{PROP: sigma-k}. The classical Alexandrov theorem (see \cite[Theorem 6.9]{Evans-Gariepy}) states that any convex function are twice differentiable. Later, Chaudhuri-Trudinger \cite{Chaudhuri-Trudinger} extended this to $k$-convex functions with $k> n/2$. For $k\leq n/2$, twice differentiability almost everywhere may not hold for general $k$-convex functions. However, this this conclusion can hold if we impose an additional equation. See \cite[Proposition 4.1]{Shankar-Yuan-Ann-2025} for 2-convex solutions to the $\sigma_{2}$ equation $\sigma_{2}=1$. We now prove this almost everywhere twice differentiability result for solutions to the $\sigma_{k}$ equation $\sigma_{k}=f$.
    
    \begin{lemma}[Chaudhuri-Trudinger]\label{LEMMA: Hessian as Radon measures}
        Let $u$ be a $2$-convex function on $B_{1}$. Then in the distributional sense, the Hessian of $u$ can be interpreted as Radon measures $[D^2u]=[\mu^{ij}]$ with $\mu^{ij}=\mu^{ji}$, i.e.
        \[  \int u\partial^2_{ij}\varphi= \int \varphi\, \mathrm{d}\mu^{ij}\qquad \text{for any}\ \varphi\in C_{c}^{\infty}(B_{1}).  \]
    \end{lemma}

   We also need the following gradient estimates:
   \begin{lemma}\label{LEMMA: Gradient estimate}
       For $k\ge 2$, let $u$ be a $C^3$ $k$-convex solution to $\sigma_{k}(D^2u)=f$ on $B_{R}$. Suppose $\psi= f^{1/k}$ is Lipschitz. Then we have
       \begin{align}\label{eq: weighted gradient estimate}
        \sup_{\substack{x,y\in B_{R}\\ x\neq y}}d_{x,y}^{n+1}\dfrac{|u(x)-u(y)|}{|x-y|}\leq C(n)\left( \int_{B_{R}}|u|\,\mathrm{d}x+R^{n+3}\sup_{B_{R}}|D\psi| \right),
        \end{align}
        where $d_{x,y}=\min\{d_{x},d_{y}\}$, and $d_{x}=\mathrm{dist}(x,\partial B_{R})=R-|x|$.
   \end{lemma}

    \begin{remark}\label{RMK: Lipschitz Reg}
        By solving the Dirichlet problem with smooth approximating data \cite{CNS3}, we can show that $k$-convex viscosity solutions to $\sigma_{k}=f$ are locally Lipschitz with estimate (\ref{eq: weighted gradient estimate}).
    \end{remark}

    \begin{proof}
        We first introduce some notations, which is same as \cite[P585]{Trudinger-Wang-Annals}. Denote
        \begin{align*}
            |u|_{0;R}^{(n)}=\sup_{x\in B_{R}}d_{x}^{n}|u(x)|,\qquad [u]_{0,1;R}^{(n)}= \sup_{\substack{x,y\in B_{R}\\ x\neq y}}d_{x,y}^{n+1}\dfrac{|u(x)-u(y)|}{|x-y|}.
        \end{align*}
        There is the following interpolation inequality between these two norms: for any $\varepsilon>0$, 
        \begin{equation}\label{eq: interpolation}
            |u|_{0;R}^{n}\leq \varepsilon [u]_{0,1;R}^{(n)}+C(n)\varepsilon^{-n}\int_{B_{R}}|u|
        \end{equation}

    Fix any $x,y\in B_{R} $ with $x\neq y$. Denote $d=d_{x,y}$. For any $t\in[0,1]$, set $z_{t}=tx+(1-t)y$. We have $B_{d}(z_{t})\subset B_{R}$. In particular, $B_{d/2}(z_{t})\subset B_{R-d/2}$. 

    From the fundamental theorem of calculus,
    \begin{equation}\label{eq: esti |u(x)-u(y)|}
         \begin{split}
             |u(x)-u(y)|=\left| \int_{0}^{1} \dfrac{\mathrm{d}}{\mathrm{d}t}u(z_{t})\,\mathrm{d}t \right|\leq \int_{0}^{1} |\nabla u(z_{t})|\,\mathrm{d}t\, |x-y|.
         \end{split}
    \end{equation}
    Applying Trudinger's gradient estimate \cite[P1258]{Trudinger-CPDE} in $B_{d/2}(z_{t})$, we obtain
    \begin{align*}
        |\nabla u(z_{t})|&\leq C(n)\left(\dfrac{1}{d}\sup_{B_{d/2}(z_{t})}|u|+d^2\sup_{B_{d/2}(z_{t})}|D\psi|\right)\\
        &\leq C(n)\left(\dfrac{1}{d}\sup_{B_{R-d/2}}|u|+d^2\sup_{B_{R}}|D\psi|\right).
    \end{align*}
    Then from \eqref{eq: esti |u(x)-u(y)|}, we have
    \begin{align*}
        d^{n+1}\dfrac{|u(x)-u(y)|}{|x-y|}&\leq C(n)\left( d^n\sup_{B_{R-d/2}}|u|+d^{n+3}\sup_{B_{R}}|D\psi| \right)\\
        &\leq C(n)\left(|u|_{0;R}^{n}+R^{n+3}\sup_{B_{R}}|D\psi| \right).
    \end{align*}
    Taking the supremum over $x,y$ and combining with the interpolation inequality \eqref{eq: interpolation} with $\varepsilon=1/2C(n)$, we finish the proof.
\end{proof}

\noindent\emph{Proof of (i) in Proposition \ref{PROP: sigma-k}.} \emph{Step 1. Approximation in the $L^{1}$ sense.}   By Lemma
\ref{LEMMA: Hessian as Radon measures}, the $k$-convexity of $u$ implies that the distributional Hessian of $u$ can be interpreted as  Radon measures $[D^2u]=[\mu^{ij}]$.  From the Lebesgue-Radon-Nikodym decomposition, we write  $\mu^{ij}=u^{ij}\,\mathrm{d}x+\mu^{ij}_{s}$, where $\mathrm{d}x$ denotes the $n$-dimensional Lebesgue measure, $u^{ij}\in L^{1}_{\mathrm{loc}}$ denotes the absolutely continuous part with respect to $\mathrm{d}x$, and $\mu^{ij}_{s}$ denotes the singular part. Write $[D^2u]=D^2u\, \mathrm{d}x+[D^2u]_{s}$, where $D^2u=(u^{ij})$ and $[D^2u]_{s}=[\mu^{ij}_{s}]$. For almost every $x\in B_{1}$, there hold
\begin{align}
    &\lim_{r\to 0}\fint_{B_{r}(x)}|D^2u(y)-D^2u(x)|\,\mathrm{d}y=0, \label{eq: 6.1} \\
    &\lim_{r\to 0}\dfrac{1}{r^n}\|[D^2u]_{s}\|(B_{r}(x))=0. \label{eq: 6.2}
\end{align}
Here $\|[D^2u]_{s}\|$ denotes the total variation of $[D^2u]_{s}$. By Remark \ref{RMK: Lipschitz Reg}, we know that $u$ is locally Lipschitz.  Rademacher's theorem tells us that $u$ is differentiable almost everywhere in $B_{1}$. In particular, for almost every $x\in B_{1}$, we also know that
\begin{equation}\label{eq: 6.3}
    \lim_{r\to 0}\fint_{B_{r}(x)}|Du(y)-Du(x)|\,\mathrm{d}y=0.
\end{equation}

Fix any $x$ such that $(\ref{eq: 6.1}) (\ref{eq: 6.2})$ and $(\ref{eq: 6.3})$ hold, we will show that $h(y)=o(|y-x|^2)$ as $y\to x$, where
\[ h(y)=u(y)-u(x)-Du(x)\cdot(y-x)-\dfrac{1}{2}(y-x)^{T}D^2u(x)(y-x). \]

 Following verbatim Steps 2–4 in the proof of Theorem 6.9 from \cite[P274-275]{Evans-Gariepy}, we can conclude the $L^{1}$ approximation:
\[ \fint_{B_{r}(x)}|h(y)|\,\mathrm{d}y=o(r^2),\qquad \mathrm{as}\ r\to0. \]

\emph{Step 2. Lipschitz estimate.} Fix any $r$ satisfying $0<2r<1-|x|$. Consider $g(y)=u(y)-u(x)-Du(x)\cdot(y-x)$. It is also a $k$-convex viscosity solution to the equation $\sigma_{k}(D^2g)=f$ on $B_{2r}(x)$. Since $f>0$ is Lipschitz, we can assume that $f$ has positive lower bound locally. For example, denote $f_{0}=\min_{\overline{B}_{1-|x|}(x)} f>0$, then $\psi=f^{1/k}$ is also Lipschitz in $B_{1-|x|}(x)$ with $\|D\psi\|_{L^{\infty}}\leq C(\|f\|_{C^{0,1}}, f_{0})$. Now, applying Lemma \ref{LEMMA: Gradient estimate}, we have
\begin{equation}
    \begin{split}
        r^{n+1}\sup_{\substack{y,z\in B_{r}(x)\\ y\neq z}}\dfrac{|g(y)-g(z)|}{|y-z|}&\leq \sup_{\substack{y,z\in B_{2r}(x)\\ y\neq z}}d_{y,z}^{n+1}\dfrac{|g(y)-g(z)|}{|y-z|}\\
        &\leq C\left(\int_{B_{2r}(x)}|g(y)|\,\mathrm{d}y+r^{n+2}\right)\\
        &\leq C\int_{B_{2r}(x)}|h(y)|\,\mathrm{d}y+ Cr^{n+2}
    \end{split}
\end{equation}
where $C$ depends on $n, \|D\psi\|_{L^{\infty}}$ and $|D^2u(x)|$. We emphasize that $C$ is independent on $r$. Note that
\[ (y-x)^{T}D^2u(x)(y-x)-(z-x)^{T}D^2u(x)(z-x)=(y+z-2x)^{T}D^2u(x)(y-z). \]
Then, we obtain the following Lipschitz estimate:
\begin{equation}\label{eq:6.5}
    \begin{split}
        \sup_{\substack{y,z\in B_{r}(x)\\ y\neq z}}\dfrac{|h(y)-h(z)|}{|y-z|}&\leq \sup_{\substack{y,z\in B_{r}(x)\\ y\neq z}}\dfrac{|g(y)-g(z)|}{|y-z|}+Cr\\
        &\leq \dfrac{C}{r}\fint_{B_{2r}(x)}|h(y)|\,\mathrm{d}y+Cr.
    \end{split}
\end{equation}

\emph{Step 4. Improve $L^{1}$ approximation to $L^{\infty}$.} For any small $\varepsilon>0$, we will find $r_{0}$ small, such that
\[ \dfrac{1}{r^2}\sup_{B_{r/2}(x)}|h|\leq 2\varepsilon \qquad \mathrm{for}\ r<r_{0}.   \]
Take $0<\eta<1/2$ small to be fixed later. By Step 2, we have
\begin{equation}\label{eq: 6.6}
    \mathcal{L}^{n}\left(\{z\in B_{r}(x): |h(z)|\ge\varepsilon r^2\}\right)\leq \dfrac{1}{\varepsilon r^2}\fint_{B_{r}(x)}|h(z)|\,\mathrm{d}z\leq \dfrac{o(r^{2})}{\varepsilon r^2}<\dfrac{1}{2}\eta^n \mathcal{L}^n(B_{r}),
\end{equation}
provided $r<r_{0}=r_{0}(n,\eta,\varepsilon)$ small. We claim that for any $y\in B_{r/2}(x)$, there exists $z\in B_{\eta r}(y) $ such that $|h(z)|<\varepsilon r^2$. Otherwise, we have $B_{\eta r}(y)\subset \{|h|\ge \varepsilon r^2\}\cap B_{r}(x)$, this implies
\[ \eta^n\mathcal{L}^{n}(B_{r})\leq \mathcal{L}^{n}\left(\{z\in B_{r}(x): |h(z)|\ge\varepsilon r^2\}\right)<\dfrac{1}{2}\eta^n \mathcal{L}^{n}(B_{r}). \]
It contradicts with (\ref{eq: 6.6}).  Therefore,
\begin{align*}
    |h(y)|\leq |h(z)|+\eta r \dfrac{|h(y)-h(z)|}{|y-z|}&\leq \varepsilon r^2+\eta r\left(\dfrac{C}{r}\fint_{B_{2r}(x)}|h|+Cr\right)\\
    &\leq \varepsilon r^2+\eta o(r^2)+ C\eta r^2. 
\end{align*}
We choose $\eta=\varepsilon/2C$ and reduce $r_{0}$ if needed, then we finally conclude that
\begin{align*}
    |h(y)|\leq  2\varepsilon r^2\qquad \mathrm{for}\ r<r_{0}=r_{0}(n,\varepsilon).
\end{align*}
Since $y\in B_{r/2}(x)$ is arbitrary, we finish the proof.\hfill\qed

Combining this almost everywhere twice differentiability result and our generalized small perturbation theorem, the partial regularity assertion in Proposition \ref{PROP: sigma-k} follows.
\ \\

\noindent\emph{Proof of (ii) in Proposition \ref{PROP: sigma-k}.} Fix any $0<\alpha<1$. Since $u$ is twice differentiable almost everywhere in $B_{1}$. It suffices to prove that $u$ is $C^{2,\alpha}$ near such twice differentiable points.
        
        For any twice differentiable point $x_{0}\in B_{1}$, there exists a quadratic polynomial $Q_{x_{0}}$, such that 
        \[ |u(x)-Q_{x_{0}}(x)|=o(|x-x_{0}|^2),\quad \text{as}\ x\to x_{0}. \]
        For $r>0$ small to be fixed later, consider the rescaled function 
        \[ v(x)=\dfrac{u(x_{0}+rx)-Q_{x_{0}}(x_{0}+rx)}{r^2},\quad \text{for}\ x\in B_{1}. \]
        Then $v$ is a viscosity solution to $G(D^2v)=\widetilde{f}$ on $B_{1}$, where
        \begin{align*}
            G(M)&=\sigma_{k}(M+D^2Q_{x_{0}})-\sigma_{k}(D^2Q_{x_{0}}),\\
            \widetilde{f}(x)&=f(x_{0}+rx)-f(x_{0}).
        \end{align*}
        Clearly, $G$ satisfies the hypotheses of the small perturbation Theorem \ref{theorem: main thm}. Let $\delta, C$ be the constants in Theorem \ref{theorem: main thm}. Now,
        \begin{align*}
            \|v\|_{L^{\infty}(B_{1})}&=\dfrac{\|u-Q_{x_{0}}\|_{L^{\infty}(B_{r}(x_{0}))}}{r^2}=\dfrac{o(r^2)}{r^2}<\delta,\\
            \|\widetilde{f}\|_{C^{0,\alpha}(B_{1})}&\leq \sup_{B_{r}(x_{0})}|f(x)-f(x_{0})|+r^{\alpha} [f]_{C^{0,\alpha}(B_{r}(x_{0}))}<\delta,
        \end{align*}
        provided $r$ is small. Fix this small $r$,  Theorem \ref{theorem: main thm} implies that $v\in C^{2,\alpha}(B_{1/2})$, then $u\in C^{2,\alpha}(B_{r/2}(x_{0}))$. The proof is now complete. \hfill\qed

\bibliographystyle{amsalpha}
\bibliography{ref}
\end{document}